\documentclass[11pt]{amsart}
\usepackage{amsmath,amssymb,amsbsy,amsfonts,amsthm,latexsym,amsopn,amstext,
                                                    amsxtra,euscript,amscd,color}
\usepackage[english]{babel}
\begin{document}

\newtheorem{thm}{Theorem}
\newtheorem*{unnumberedthm}{Theorem}
\newtheorem{lemma}[thm]{Lemma}
\newtheorem{claim}[thm]{Claim}
\newtheorem{corollary}[thm]{Corollary}
\newtheorem{prop}[thm]{Proposition} 
\newtheorem*{definition}{Definition}
\newtheorem{question}[thm]{Open Question}
\newtheorem{conj}[thm]{Conjecture}
\newtheorem{prob}{Problem}

\newcounter{app}\numberwithin{app}{section}
\newtheorem{applemma}[app]{Lemma}
\newtheorem{appcorollary}[app]{Corollary}
\def\vol {{\mathrm{vol\,}}}
\def\squareforqed{\hbox{\rlap{$\sqcap$}$\sqcup$}}
\def\qed{\ifmmode\squareforqed\else{\unskip\nobreak\hfil
\penalty50\hskip1em\null\nobreak\hfil\squareforqed
\parfillskip=0pt\finalhyphendemerits=0\endgraf}\fi}

\def\cA{{\mathcal A}} \def\cB{{\mathcal B}} \def\cC{{\mathcal C}}
\def\cD{{\mathcal D}} \def\cE{{\mathcal E}} \def\cF{{\mathcal F}}
\def\cG{{\mathcal G}} \def\cH{{\mathcal H}} \def\cI{{\mathcal I}}
\def\cJ{{\mathcal J}} \def\cK{{\mathcal K}} \def\cL{{\mathcal L}}
\def\cM{{\mathcal M}} \def\cN{{\mathcal N}} \def\cO{{\mathcal O}}
\def\cP{{\mathcal P}} \def\cQ{{\mathcal Q}} \def\cR{{\mathcal R}}
\def\cS{{\mathcal S}} \def\cT{{\mathcal T}} \def\cU{{\mathcal U}}
\def\cV{{\mathcal V}} \def\cW{{\mathcal W}} \def\cX{{\mathcal X}}
\def\cY{{\mathcal Y}} \def\cZ{{\mathcal Z}}

\def\NmQR{N(m;Q,R)} \def\VmQR{\cV(m;Q,R)}

\def\Xm{\cX_m}

\def \A {{\mathbb A}} \def \B {{\mathbb A}} \def \C {{\mathbb C}} \def \F
{{\mathbb F}} \def \G {{\mathbb G}} \def \L {{\mathbb L}} \def \K {{\mathbb K}}
\def \Q {{\mathbb Q}} \def \R {{\mathbb R}} \def \Z {{\mathbb Z}} \def
\fS{\mathfrak S}


\def\\{\cr} \def\({\left(} \def\){\right)}
\def\fl#1{\left\lfloor#1\right\rfloor} \def\rf#1{\left\lceil#1\right\rceil}

\def\Tr{{\mathrm{Tr}}} \def\Im{{\mathrm{Im}}}

\def \bFp {\overline \F_p}

\newcommand{\pfrac}[2]{{\left(\frac{#1}{#2}\right)}}

\def \Prob{{\mathrm {}}} \def\e{\mathbf{e}} \def\ep{{\mathbf{\,e}}_p}
\def\epp{{\mathbf{\,e}}_{p^2}} \def\em{{\mathbf{\,e}}_m}

\def\Res{\mathrm{Res}} \def\Orb{\mathrm{Orb}}

\def\vec#1{\mathbf{#1}} \def\flp#1{{\left\langle#1\right\rangle}_p}

\def\mand{\qquad\mbox{and}\qquad}

\newcommand{\commA}[1]{\marginpar{%
    \vskip-\baselineskip 
    \itshape\hrule\smallskip\begin{color}{red}#1\end{color}\par\smallskip\hrule}}

\newcommand{\commO}[1]{\marginpar{%
    \vskip-\baselineskip 
    \itshape\hrule\smallskip\begin{color}{blue}#1\end{color}\par\smallskip\hrule}}

\title[Perpendicular bisectors and pinned distances]{On distinct perpendicular
bisectors and pinned distances in finite fields}

 \author[B. Hanson, B. Lund and O. Roche-Newton]{Brandon Hanson, Ben Lund and
 Oliver Roche-Newton}

\address{Department of Mathematics, University of Toronto, Ontario, Canada} \email{bhanson@math.utoronto.ca}

\address{Department of Computer Science, Rutgers, The State University of New Jersey, NJ} \email{lund.ben@gmail.com}

\address{Johann Radon Institute for Computational and Applied Mathematics,
Austrian Academy of Sciences, 4040
 Linz, Austria} \email{o.rochenewton@gmail.com}

\date{\today}

\newcommand{\NN}[0]{\mathbb N}
\newcommand{\FF}[0]{\mathbb F}
\newcommand{\EE}[0]{\mathbb E}
\newcommand{\PP}[0]{\mathbb P}
\newcommand{\uu}[0]{\textbf{\textit u}}
\newcommand{\vv}[0]{\textbf{\textit v}}
\renewcommand{\aa}[0]{\textbf{\textit a}}
\newcommand{\bb}[0]{\textbf{\textit b}}
\newcommand{\cc}[0]{\textbf{\textit c}}
\newcommand{\dd}[0]{\textbf{\textit d}}
\newcommand{\xx}[0]{\textbf{\textit x}}
\newcommand{\yy}[0]{\textbf{\textit y}}
\newcommand{\zz}[0]{\textbf{\textit z}}
\newcommand{\ww}[0]{\textbf{\textit w}}
\newcommand{\ii}[0]{\textbf{\textit i}}
\newcommand{\jj}[0]{\textbf{\textit j}}
\newcommand{\mm}[0]{\textbf{\textit m}}
\renewcommand{\dd}[0]{\textbf{\textit d}}
\newcommand{\eps}[0]{\varepsilon}
\newcommand{\mat}[4]{\left( \begin{array}{cc} #1 & #2 \\ #3 & #4 \end{array} \right)}
\renewcommand{\vec}[2]{\left( \begin{array}{c} #1 \\ #2 \end{array} \right)}
\newcommand{\one}[0]{\textbf{1}}

\date{\today}

\begin{abstract}
Given a set of points  $P \subset \F_q^2$ such that $|P|\geq
q^{4/3}$, we establish that for a positive proportion of points $\aa \in P$, we have
 $$|\{\|\aa-\bb\|:\bb \in P\}|\gg q,$$ where $\|\aa-\bb\|$ is the distance between points $\aa$ and $\bb$.
This improves a result of Chapman et al. \cite{Ch}.

A key ingredient of our proof also shows that, if $|P|\geq q^{3/2}$, then the number $B$ of distinct lines which arise as
the perpendicular bisector of two points in $P$ satisfies $B\gg q^2$.
\end{abstract}


\keywords{finite fields, perpendicular bisectors, pinned distances, isosceles triangles, rigid motions, expander mixing lemma}

\maketitle

\section{Introduction}

Given a set of points $P$, it is natural to construct a set of lines by
connecting distinct pairs of elements from $P$. Roughly speaking, one expects
that this set of lines determined by $P$ should be large, unless the point set
is highly collinear. A seminal result of this kind was Beck's Theorem
\cite{Beck}, which established that there exist absolute constants $c,k>0$ such
that if $P \subset \mathbb{R}^2$, then either $P$ determines $c|P|^2$ distinct
lines, or there exists a single line supporting $k|P|$ points from $P$.
Different versions of Beck's Theorem for the finite field\footnote{Whenever we refer
to a field $\F_q$ in this paper, it is assumed that the field has characteristic
strictly greater than $2$.} setting, in which we
begin by considering a point set $P \subset \F_q^2$, have been proven in \cite{Alon,HR,Jones}.

In this paper, we consider an alternative take on Beck's theorem, in which we
look at the set of perpendicular bisectors determined by a point set $P \subset
\F_q^2$. 
For a vector $\xx = (x_1,x_2) \in \F_q^2$, we define $\|\xx \| = x_1^2+x_2^2$
and call $\|\xx-\yy\|$ the \textit{distance} between $\xx$ and $\yy$. Though it
is not actually a metric since it takes values in $\FF_q$, this notion
of distance shares a number of purely algebraic properties with the Euclidean
distance. Given two distinct points $\aa,\bb \in \F_q^2$, we define the
\textit{perpendicular bisector of $\aa$ and $\bb$} to be the set $$B(\aa,\bb):=\{\cc \in \F_q^2: \|\cc-\aa\|=\|\cc-\bb\|\}.$$
 It is a simple calculation to check
that $B(\aa,\bb)$ is indeed a line in $\F_q^2$. Now, define $B(P)$ to be the set
of all perpendicular bisectors determined by pairs of points from $P$ with non-zero distance. That is,
$$B(P):=\{B(\aa,\bb):\aa,\bb \in P, \|\aa-\bb\| \neq 0\}.$$
Again, we expect that $|B(P)|$ will be large, provided that $P$ is not of some
degenerate form. One of these degenerate cases occurs when the
point set $P$ consists of many points on the same line. If all of the
points lie on the same line, it is possible that $|B(P)|$ could be as small as
$2|P|-3$. Another degenerate case occurs when the points of $P$ are
equidistributed on a circle. However, these constructions only seem to work for
relatively small point sets. Indeed the points on a line or circle are
contained in a one-dimensional subset of the plane. In this paper we prove that
provided a point set is sufficiently large, a positive
proportion of all lines arise as perpendicular bisectors:

\begin{thm} \label{main1} If $P\subset \F_q^2$ such that $|P|\geq q^{3/2}$, then
$$|B(P)|\gg q^2.$$
\end{thm}

In concurrent work, Lund, Sheffer and de Zeeuw proved an analog to Theorem \ref{main1} for finite sets of points in the real plane \cite{LSZ14}.

Theorem \ref{main1} was partly motivated by an application for the ``pinned
distance problem'' in $\F_q^2$. The aim of this problem is to show that, for
any given point set, there always exists a point (or indeed many points) from
the set which determines many distances with the rest of the point set. One of
the key ingredients used to prove Theorem \ref{main1} can be combined with an
incidence theorem for multisets of points and lines in order to establish the
following result:

\begin{thm} \label{main2} Let $P \subset \F_q^2$ such that $|P| \geq q^{4/3}$.
Then there exists a subset $P' \subset P$ such that $|P'| \gg |P|$ and for all
$\aa \in {P'}$ we have the estimate
\begin{equation}
|\{\|\aa-\bb\|:\bb \in P\}|\gg q.
\label{pinned}
\end{equation}
\end{thm}

To give some context for Theorem \ref{main2}, we refer to the work of Chapman et
al. (see \cite[Theorem 2.3]{Ch}), who proved that for a point set $P\subset
\F_q^d$ with $|P| \geq q^{\frac{d+1}{2}}$, there exists a subset $P' \subset P$
such that $|P'| \gg |P|$ and for all $\xx \in P'$, \eqref{pinned} holds. Theorem \ref{main2}
gives an improvement on this result in the case when $d=2$.

The pinned distance problem is a variant on the classical Erd\H{o}s ``distinct
distance problem'', and a trivial observation is that a lower bound for pinned
distances implies a lower bound for the number of distinct distances determined
by a set. In the real plane, the distinct distance problem was almost completely
resolved by Guth and Katz \cite{GK}, whilst the harder pinned distance problem
remains wide open. The sequence of works in \cite{BHIPR, IR} established that a set
of $q^{4/3}$ points in $\mathbb{F}_q^2$ determines a positive proportion of the $q$ distinct
distances. With Theorem \ref{main2}, the threshold for the number of points in the plane that
will necessarily determine a positive proportion of all pinned distances now matches
that known for distances.

Both Theorem \ref{main1} and Theorem \ref{main2} are deduced from a bound on the
number of pairs of pairs of points that determine the same line as a bisector.
For a point set $P \subseteq \mathbb{F}_q^2$, define the set
\[Q(P) := \{(\xx,\yy, \zz, \ww) \in P^4 : B(\xx,\zz) = B(\yy,\ww), \| \xx - \zz \| \neq 0 \}.\]

When the point set $P$ is obvious from the context, we will sometimes drop the
argument and simply write $Q$ instead of $Q(P)$.

\begin{thm} \label{energybound}
For $P \subset \F_q^2$,
 \[|Q(P)| \ll \frac{|P|^4}{q^2}+q|P|^2.\]
\end{thm}

It is interesting to note that, while we don't believe that Theorems \ref{main1}
or \ref{main2} are tight, Theorem \ref{energybound} is tight up to the implicit
constants.
A randomly chosen set of points shows that it is possible to construct a set $P
\subset \F_q^2$ for which $|Q| \gg |P|^4/q^2$, and hence Theorem
\ref{energybound} is tight when $|P| \gg q^{3/2}$.
Suppose now that $P$ consists of the union of $|P|/q$ parallel lines, each
containing $q$ points. Now let $l$ be a line perpendicular to these lines. Then
$l$ is the bisector of $|P|$ pairs of points in $P$ - each line contains $q$
pairs of points with $l$ as their bisector. In particular, the number of pairs
$(\xx,\zz)$ and $(\yy,\ww)$ each with $l$ as their bisector is $|P|^2$. There
are $q$ lines $l$ which are perpendicular to the lines defining $P$, so $|Q|
\geq |P|^2q$; this shows that Theorem \ref{energybound} is tight when $|P| \ll
q^{3/2}$.

Note that in the definitions of the sets $B(P)$ and $Q(P)$, we make the point of
excluding pairs of points whose distance is zero.
These can arise if the element $-1\in\FF_q$ has a square root in $\mathbb{F}_q$,
which happens exactly when $q$ is congruent to $1$ modulo $4$.
The possibility of points with zero distances present some extra technical
difficulties in the forthcoming analysis, and it is often necessary to consider
separately the cases when $q$ is congruent to either $1$ or $3$ modulo $4$. In
fact, Theorem \ref{energybound} would not be true if quadruples arising from
such zero distances were included in the set $Q$. It can be calculated that if
$P$ is the union of  $|P|/q$ isotropic lines, then we would have $|P|q^3$
quadruples since any four points $\xx,\yy,\zz,\ww$ on an isotropic line satisfy
$B(\xx,\zz)=B(\yy,\zz)$. The arguments in this paper when $q=3 \mod 4$ are
slightly more straightforward, since zero distances are not an issue in this
case.

The rest of the paper is structured as follows. In section \ref{Prelims} we go
over preliminary results needed. We cover some simple plane geometry over finite
fields in subsection \ref{sec:finitePlaneGeometry}. We will quote a number of
elementary results from finite plane geometry, though we leave the proofs to an
appendix in the interest of brevity. If the reader is familiar with plane
geometry, these results will be quite believable. In subsection \ref{sec:EML},
we record a version of the Expander Mixing Lemma, also proved in an appendix,
and recall a few facts we will need from linear algebra. We prove Theorem \ref{energybound}, and consequently Theorem \ref{main1}, in section \ref{sec:bisectors}. 
 Finally, in section \ref{sec:pinnedDistances}, we record
a version of the finite field Szemer\'{e}di-Trotter theorem for multisets, and
combine this incidence result with Theorem \ref{energybound} in order to prove
Theorem \ref{main2}.

The methods used are a combination of elementary geometry and spectral graph theory. It is likely that Fourier analysis (i.e.
the use of exponential sums) would succeed just as well in proving our theorems,
but we have instead chosen to work with graphs, which maintains the combinatorial
spirit of the problem.

\subsubsection*{Notation}
We recall that the notations $U \ll V$ and  $V \gg U$ are both equivalent to the
statement that the inequality $|U| \le c V$ holds with some
constant $c> 0$.

\section{Preliminaries}\label{Prelims}

\subsection{Results from Finite Plane Geometry}\label{sec:finitePlaneGeometry}
In this section we establish a few facts from finite planar geometry. For proofs of the facts claimed in this section we refer to Appendix
\ref{sec:GeometryAppendix}.

Recall the notion of distance introduced earlier; when
$\xx=(x_1,x_2),\yy=(y_1,y_2)\in\FF_q^2$ we will write \[\xx\cdot \yy=\xx^t\yy\]
for the standard inner product and \[\|\xx\|=\xx\cdot\xx.\] This is not a
distance in usual sense since the elements of $\FF_q$ are not sensibly ordered,
but many of properties of a norm persist in an algebraic fashion.
The set of points a fixed distance from a given point is a circle
\[C_r(\uu)=\{\xx\in\FF_q^2:\|\xx-\uu\|=r\}.\]
We call $\uu$ the centre of the circle and $r$ the radius. We remark here that
when $q=1\mod 4$ then there is an element $i\in\FF_q$ satisfying $i^2=-1$ and so
there are non-zero points on the circle of zero radius. Recall that the bisector of two
distinct points $\xx$ and $\yy$, is denoted as \[B(\xx,\yy):=\{\cc \in \F_q^2:
\|\cc-\xx\|=\|\cc-\yy\|\}.\] Equivalently, the bisector of $\xx$ and $\yy$ is
the line passing through the the midpoint $\frac{1}{2}(\xx+\yy)$ with direction
orthogonal to $\xx-\yy$.

We will use the symmetries of the plane to understand the distribution of
bisectors.

\begin{definition}[Rotations, Reflections and Translations] A matrix of the form
\[\mat{a}{-b}{b}{a},\ a^2+b^2=1\] is called a rotation matrix, and a matrix of
the form \[\mat{a}{b}{b}{-a},\ a^2+b^2=1\] is called a reflection matrix. If
$\uu\in\FF_q^2$ and $R$ is a rotation matrix, then a rotation about $\uu$ by $R$
is an affine map of the form \[\cR(\vv)=\cR_{R,\uu}(\vv)=R(\vv-\uu)+\uu.\] If
$\uu\in\FF_q^2$ and $S$ is a reflection matrix, then a reflection about $\uu$ by $S$
is an affine map of the form \[\cS(\vv)=\cS_{S,\uu}(\vv)=S(\vv-\uu)+\uu.\]
A translation by $\uu$ is an affine map of the form
\[\cT(\vv)=\cT_{\uu}(\vv)=\vv+\uu.\]
\end{definition}

For our purposes we need the connection between reflections and bisectors. However, it is not the case that all lines arise as the fixed line of a reflection. Indeed, we need to account for the possibility of elements with vanishing norm. A line $l$ is called \textit{isotropic} if it is of the form \[l=\{t\cdot (1,\pm
i)+\uu:t\in \FF_q\}\] where $i^2=-1$. With these definitions, we are now ready to state four lemmas which will be used towards proving the results on distinct bisectors and pinned distances.

\begin{lemma}\label{th:distanceDecomposition}
If $\xx,\yy,\zz,\ww \in \FF_q^2$ are such that $B=B(\xx,\zz)=B(\yy,\ww)$ is
non-isotropic then $\|\xx-\yy\|=\|\zz-\ww\|$.
\end{lemma}

\begin{lemma}
\label{th:reflections}
There are $q+1$ lines passing through any $\uu\in \FF_q^2$. 
\begin{enumerate}
	\item If $q=1\mod 4$ then two of the lines are isotropic and $q-1$ are
	non-isotropic. It follows that there are $q-1$ rotations and
	reflections about any point $\uu$.
	\item When $q=3\mod 4$ all lines are non-isotropic. It follows that there are $q+1$ rotations and
	reflections about any point $\uu$.
\end{enumerate}
\end{lemma}

\begin{lemma}\label{th:circles}
Suppose $\uu\in\FF_q^2$ and $r\in \FF_q$. Then we have:
\begin{enumerate}
  \item $|C_r(\uu)|=q-1$ if $r\neq 0$ and $|C_0(\uu)|=2q-1$ whenever $q=1\mod 4$;
  \item $|C_r(\uu)|=q+1$ if $r\neq 0$ and $|C_0(\uu)|=1$ whenever $q=3\mod 4$.
\end{enumerate} It follows
that the number of ordered pairs $(\xx,\yy)\in\FF_q^2\times\FF_q^2$ with
$\|\xx-\yy\|=r$ is:
\begin{enumerate}
  \item $q^2(q-1)$ if $r\neq 0$ and $q^2(2q-1)$ if $r=0$ whenever $q=1\mod 4$;
  \item $q^2(q+1)$ if $r\neq 0$ and $q^2$ if $r=0$ whenever $q=3\mod 4$.
\end{enumerate}
\end{lemma}

\begin{lemma}\label{th:rotationsFromReflections2}
Suppose $\xx,\yy,\zz,\ww\in\FF_q^2$ are such that $(\xx,\yy)\neq (\zz,\ww)$ and
\[\|\xx-\yy\|=\|\zz-\ww\| \neq 0.\] 
If $\xx-\yy\neq \zz-\ww$, then there are
$q-1$ pairs of reflections $(\cR_1,\cR_2)$ with $\cR_1(\xx)=\cR_2(\zz)$ and
$\cR_1(\yy)=\cR_2(\ww)$ when $q=1\mod 4$ and $q+1$ such pairs when $q=3\mod 4$.
If $\xx-\yy=\zz-\ww$ and $\| \xx-\zz \| \neq 0$, then there are $q$ such pairs of reflections.
If $\xx-\yy=\zz-\ww$ and $\| \xx-\zz \| = 0$, then there are no such pairs of reflections.
\end{lemma}

\subsection{Some Linear Algebra and the Expander Mixing Lemma}\label{sec:EML}
Here, we recall some simple facts we need from linear algebra. The main tool we use
in our arguments is the Expander Mixing Lemma, a standard result in
spectral graph theory \cite{mixing}. In fact we need a
weighted variant of the lemma for the results coming in Sections
\ref{sec:bisectors} and \ref{sec:pinnedDistances}.

Suppose $G$ is a $\delta$-regular graph, meaning each vertex in $G$ is adjacent
to $\delta$ other vertices. If $A$ is the adjacency matrix of $G$, note that the
largest eigenvalue of $A$ is $\delta$; the eigenvector corresponding to this
eigenvalue is the all-$1$s vector.
We let $L^2(V)$ be the set of complex valued functions on the vertex set $V$
endowed with the inner product \[\langle f,g\rangle=\sum_{v\in V}f(v)\bar g(v)\]
and norm \[\|f\|^2=\langle f,f\rangle.\] The matrix $A$ acts on $L^2(V)$ by the
formula \[Af(v)=\sum_{\{u,v\}\in E} f(u).\] Finally, we let $\EE$ denote the
expectation:
\[\EE(f)=\frac{1}{|V|}\sum_{v\in V}f(v).\]

We recall here the following versions of the Plancherel and Parseval identities.

\begin{lemma}\label{Fourier}
Let $B$ be an orthonormal basis for $L^2(V)$. Then we have 
\[\sum_{v\in V}|f(v)|^2=\sum_{b\in B}|\langle
  f,b\rangle|^2\] and
\[\sum_{v\in V}f(v)\bar{g}(v)=\sum_{b\in B}\langle
  f,b\rangle\langle b,g\rangle.\]
\end{lemma}

We have the following version of the expander mixing lemma. The proof is postponed to Appendix \ref{app2}.
\begin{lemma}[Expander Mixing Lemma]\label{th:expanderMixing}
Let $G = (V,E)$ be a $\delta$-regular graph with $|V|=n$, and let $A$ be the
adjacency matrix for $G$. Suppose the absolute values of all but the largest
eigenvalue of $A$ are bounded by $\lambda$.
Suppose $f,g\in L^2(V)$, then \[\left|\langle
f,Ag\rangle-\delta n\EE(f)\EE(g)\right|\leq \lambda\|f\|\|g\|.\] In particular,
let $S,T \subseteq V$, and denote by $E(S,T)$ the number of edges between $S$ and $T$.
Then,
\[\left | E(S,T) - \delta|S| |T| / n \right | \leq \lambda \sqrt{|S| |T|}.\]
\end{lemma}

Finally we will also need the following standard fact from linear algebra.
\begin{lemma}[Gershgorin Circle Theorem]\cite{Brualdi}\label{th:Gershgorin}
Let $A=[A_{ij}]$ be an $n \times n$ matrix, and let $r_i = \sum_{j=1}^n
|a_{ij}|$ be the sum of the absolute values of the $i^{th}$ row of $A$.
Then each eigenvalue of $A$ is contained in at least one of the disks
\[\mathcal{D}_i = \{z:|z - a_{ii}| \leq r_i\}\]
in the complex plane.
\end{lemma}

\section{Distinct perpendicular bisectors}\label{sec:bisectors}

In this section, we will prove Theorems \ref{main1} and \ref{energybound}.
The basic approach is to associate our problem with a graph, and then use the
facts about rigid motions from Section \ref{sec:finitePlaneGeometry} to analyze
the eigenvalues of this graph, so that we can apply the expander mixing lemma.

It will be more convenient to deduce Theorem \ref{energybound} from 
the following similar result:

\begin{lemma} \label{energybound2}
For a point set $P \subseteq \mathbb{F}_q^2$, define the set
\[Q'(P) := \{(\xx,\yy, \zz, \ww) \in P^4 : B(\xx,\zz) = B(\yy,\ww), \| \xx - \yy \| \neq 0 \}.\]
Then
\[|Q'(P)| \ll \frac{|P|^4}{q^2}+q|P|^2.\]
\end{lemma}

Note that, although Lemma \ref{energybound2} appears very similar to Theorem
\ref{energybound}, they are not identical. In the definition of $Q(P)$
quadruples for which $\|\xx-\zz\|=0$ are excluded, whereas the definition of
$Q'(P)$ excludes quadruples for which $\|\xx-\yy\|=0$. First we deduce Theorem
\ref{energybound} from Lemma \ref{energybound2}.

\begin{proof}[Proof of Theorem \ref{energybound} from Lemma \ref{energybound2}] 
Define the sets 
\begin{align*}
Q'=Q_1&:=\{(\xx,\yy, \zz, \ww) : B(\xx,\zz) = B(\yy,\ww), \| \xx - \yy \| \neq 0
\}, \\Q_2&:=\{(\xx,\yy, \zz, \ww) : B(\xx,\zz) = B(\yy,\ww), \| \xx - \ww \|
\neq 0 \}, \\Q_3&:=\{(\xx,\yy, \zz, \ww) : B(\xx,\zz) = B(\yy,\ww), \| \zz - \yy
\| \neq 0 \}, \\Q_4&:=\{(\xx,\yy, \zz, \ww) : B(\xx,\zz) = B(\yy,\ww), \| \zz -
\ww \| \neq 0 \}.
\end{align*}
Note that $|Q_1|=|Q_2|=|Q_3|=|Q_4|$, since there is a natural bijection between $Q_i$ and $Q_j$ for any $1\leq i,j \leq 4$. Therefore,
$$|Q| \ll |Q'|+ |Q \setminus (Q_1 \cup Q_2 \cup Q_3 \cup Q_4)|.$$
It remains to bound the size of the set $Q \setminus (Q_1 \cup Q_2 \cup Q_3 \cup Q_4):=Q''$.

Let $\xx$ and $\zz$ be arbitrary elements from $P$ such that $\|\xx-\zz\| \neq
0$. We will show that there are at most two pairs $(\yy,\ww)$ such that
$(\xx,\yy,\zz,\ww) \in Q''$. Indeed, if $(\xx,\yy,\zz,\ww) \in Q''$ then we have
\begin{equation} \label{long}
\|\xx-\yy\|=\|\xx-\ww\|=\|\zz-\yy\|=\|\zz-\ww\|=0.
\end{equation}

It follows from \eqref{long} that $\yy$ and $\ww$ each lie on one of the two
isotropic lines through $\xx$. Similarly, $\yy$ and $\ww$ each lie on one of the
two isotropic lines through $\zz$. However, since $\|\xx-\zz\| \neq 0$, these
four isotropic lines are distinct. There are then only two possible choices for
the pair $(\yy,\ww)$ (including reordering the two elements).

Finally, we have $|Q''|\ll |P|^2$, and it then follows from Lemma \ref{energybound} that
$$|Q| \ll |Q'|+|Q''| \ll \frac{|P|^4}{q^2}+q|P|^2+|P|^2 \ll  \frac{|P|^4}{q^2}+q|P|^2 ,$$
as required.
\end{proof}

\subsection*{Fixed Distance Bisector Quadruples}
We are now going to prove the main technical result from which we will deduce
our results. We will split the set of bisector quadruples according to the
distances between the pairs of points. The technical result, which proceeds by way of
spectral graph theory, shows that the bisector quadruples at a given distance
are distributed uniformly among all point pairs with this distance. Then we employ an estimate for distance
quadruples and deduce a bound for bisector quadruples.

For the remainder of this section, let $P \subseteq \mathbb{F}_q^2$ be a fixed
set of points.
Recall that our immediate goal is to place an upper bound on the size of the set
\[Q' = Q'(P) = \{(\xx,\yy,\zz,\ww) \in P^4 : B(\xx,\zz) = B(\yy,\ww), \|\xx-\yy\| \neq 0\}.\]

Rather than bounding $|Q'|$ directly, we will partition $Q'$ into subsets defined by pairs of points at a fixed distance.
For each $d \in \mathbb{F}_q$, define
\begin{align*}
Q'_d 	&=  Q'_d(P) = \{(\xx,\yy,\zz,\ww) \in Q' : \lVert \xx- \yy \rVert = \|\zz-\ww\|= d\},
\text{ and} \\
\Pi_d 	&=  \Pi_d(P) = \{(\xx, \yy) \in P^2 : \Vert \xx - \yy \rVert = d\}.
\end{align*}
From Lemma \ref{th:distanceDecomposition}, we have
\[Q' = \bigcup_{d \neq 0} Q'_d.\]

The following result establishes the uniformity of bisector quadruples at
distance $d$ within all point pairs at distance $d$.

\begin{prop}\label{th:fixedDistanceEnergy}
For $d \neq 0$,
\[|Q'_d| \leq  \frac{|\Pi_d|^2}{q} + 2(q-1)|\Pi_d|.\]
\end{prop}

\begin{proof}
Let $G$ be a graph with vertices
\[V = \{(\xx,\yy) \in \mathbb{F}_q^2 \times \mathbb{F}_q^2 : \lVert \xx-\yy \rVert = d\},\]
and edges 
\[E = \{\{(\xx,\yy),(\zz,\ww)\} \in V^2 : B(\xx,\zz) = B(\yy,\ww)\}.\]

For $x \in V$, define $\Gamma(x)$ to be the neighbourhood of $x$; in other words,
\[\Gamma(x) = \{y \in V : \{x,y\} \in E\}.\]

Let $A$ be the adjacency matrix of $G$. It is straightforward to see that
$A^2_{xy} = |\Gamma(x) \cap \Gamma(y)|$, the number of paths of length $2$ from $x$ to $y$ in $G$.

The plan is to bound the second eigenvalue of $A^2$, and use this along with
Lemma \ref{th:expanderMixing} to complete the proof.
We will bound the eigenvalues of $A^2$ separately in the cases $q = 1 \mod 4$
and $q = 3 \mod 4$.
The method that we use to bound the eigenvalues in each case is reminiscent of
the method used in \cite{solymosi}.

Suppose first that $q = 1 \mod 4$.

By Lemma \ref{th:circles}, $|V| = q^2(q-1)$.
Each vertex has an edge for each of the $q(q-1)$ non-isotropic lines, so $G$ is a $q(q-1)$-regular graph.
From Lemma \ref{th:rotationsFromReflections2}, we have
\[
|\Gamma(x) \cap \Gamma(y)| =
\begin{cases}
	q(q-1), & \text{if } x = y, \\
	q, & \text{if $y$ is a non-isotropic translation of $x$}, \\
	0, & \text{if $y\neq x$ and $y$ is an isotropic translation of $x$}, \\
	q-1, & \text{otherwise}.
\end{cases}
\]

Hence, we can write
\[A^2 = (q-1)J + (q-1)^2 I + E,\]
where $J$ is the all-$1$s matrix, $I$ is the identity matrix, and $E=[E_{xy}]$ is a matrix such that
\[
E_{xy} = 
\begin{cases}
	1, & \text{if $y$ is a non-isotropic translation of $x$}, \\
	1-q, & \text{if $y\neq x$ and $y$ is an isotropic translation of $x$}, \\
	0, & \text{otherwise}.
\end{cases}
\]

Since $E$ is real and symmetric, it has real eigenvalues.
By Lemma \ref{th:circles}, any fixed pair of points has $2(q-1)$ distinct
non-trivial isotropic translations and $(q-1)^2$ distinct non-isotropic
translations.
Hence, the sum of the absolute values of the elements on each row of $E$ is
equal to $3(q-1)^2$; in other words, for any $1\leq i\leq n$, \[\sum_{j=1}^n
|E_{ij}| = 3(q-1)^2.\] Hence, by Lemma \ref{th:Gershgorin}, the absolute value
of each eigenvalue of $E$ is bounded by $3(q-1)^2$.

Since $A$ is a real symmetric matrix, its eigenvectors can be taken to be real
and orthogonal. Moreover, because the row sums of $A$ are all equal, the
all-$1$s vector is an eigenvector of $A$.

Let $\vv$ be an eigenvector of $A$ that is orthogonal to the all-$1$s vector and
that has eigenvalue $\lambda$.
It is clear that $\vv$ is an eigenvector of $I$ with eigenvalue $1$ and an
eigenvector of $J$ with eigenvalue $0$. For any constants $a,b$, the vector
$\vv$ is an eigenvector of the matrix $A^2-aI-bJ$ with eigenvalue $\lambda^2-a$.
In particular, it is an eigenvector of the matrix $E=A^2-(q-1)^2I-(q-1)J$ and
(by above) its eigenvalue is at most $3(q-1)^2$.

Now we have $E \vv = (\lambda^2 - (q-1)^2) \vv$ and hence
\begin{align*}
|\lambda^2-(q-1)^2| &\leq 3(q-1)^2,\,\,\,\,\,\,\,\,\,  \text{so that}\\
\lambda^2 & \leq  4(q-1)^2.
\end{align*}

Hence, the absolute value of each eigenvalue of $A$ that corresponds to an
eigenvector orthogonal to the all-$1$s vector is bounded above by $2(q-1)$.

Now, suppose that $q = 3 \mod 4$.

By Lemma \ref{th:circles}, $|V| = q^2(q+1)$.
Each vertex has an edge for each of the $q(q+1)$ lines, so $G$ is a $q(q+1)$-regular graph.
From Lemma \ref{th:rotationsFromReflections2}, we have
\[
|\Gamma(x) \cap \Gamma(y)| =
\begin{cases}
	q(q+1), & \text{if } x = y, \\
	q, & \text{if $y\neq x$ and $y$ is a translation of $x$}, \\
	q+1, & \text{otherwise}.
\end{cases}
\]

Hence, we can write
\[A^2 = (q+1)J + (q-1)(q+1) I + E,\]
where $J$ is the all-$1$s matrix, $I$ is the identity matrix, and $E=[E_{xy}]$ is a matrix such that
\[
E_{xy} = 
\begin{cases}
	-1, & \text{if $y\neq x$ and $y$ is a translation of $x$}, \\
	0, & \text{otherwise}.
\end{cases}
\]

By Lemma \ref{th:reflections}, any fixed pair of points has $(q-1)(q+1)$
distinct non-trivial translations.
Hence, the sum of the absolute values of the elements on each row of $E$ is
equal to $(q-1)(q+1)$; in other words, for any $1\leq i \leq n$, \[\sum_{j=1}^n
|E_{ij}| = (q-1)(q+1).\] Hence, by Lemma \ref{th:Gershgorin}, the absolute
value of each eigenvalue of $E$ is bounded by $(q-1)(q+1)$.

Let $\vv$ be an eigenvector of $A$, orthogonal to the all-$1$s vector and
associated with eigenvalue $\lambda$.
As in the case where $q= 1 \mod 4$, we have $E\vv=(\lambda^2 -(q-1)(q+1))\vv$
and so
\begin{align*}
|\lambda^2-(q-1)(q+1)| &\leq (q-1)(q+1), \\
\lambda^2 & \leq  2(q-1)(q+1).
\end{align*}

Hence, the absolute value of each eigenvalue of $A$ that corresponds to an
eigenvector orthogonal to the all-$1$s vector is bounded above by
$\sqrt{2(q+1)(q-1)} \leq 2(q-1)$.
Applying Lemma \ref{th:expanderMixing}, we have
\[E(\Pi_d,\Pi_d) \leq \frac{\delta |\Pi_d|^2}{|V|} + \lambda |\Pi_d|,\]
where $\delta$ is the degree of each vertex of $G$.

We complete the proof by observing that $E(\Pi_d,\Pi_d) = |Q'_d|$ is the exactly
the quantity that we want to bound, and then substituting the previously
calculated values for $\delta,|V|$ and $\lambda$ into this inequality.
\end{proof}

We will use the following bound on $\sum_d |\Pi_d|^2$ established in
\cite{BHIPR}\footnote{See the bound on the quantity $\sum_{\mathbb D} \mu^2(\mathbb D)$  in the proof of
Theorem 1.5 therein, in the case when $k=1$ and $d=2$.}. 
\begin{lemma}\label{th:distanceEnergyBound}
Let $P$ be a set of points in $\FF_q$. Then we have the estimate \[ \sum_d
|\Pi_d|^2 \ll |P|^4/q + q^2|P|^2.\]
\end{lemma}

\begin{proof}[Proof of Lemma \ref{energybound2}]
We use Lemma \ref{th:distanceEnergyBound} and Proposition \ref{th:fixedDistanceEnergy} to complete the proof.
\begin{align*}
|Q'| &=  	\sum_{d\neq 0} |Q'_d|, \\
	&\leq   \sum_d \left(\frac{|\Pi_d|^2}{q} + 2(q-1)|\Pi_d| \right ), \\
	&=  	\left(\sum_d  \frac{|\Pi_d|^2}{q}\right) + 2(q-1)|P|^2 , \\
	&\ll  	\frac{|P|^4}{q^2} + q|P|^2. 
\end{align*}
\end{proof}

\begin{proof}[Proof of Theorem \ref{main1}] For a line $l \in B(P)$, let
$w(l)$ be the number of point pairs $(\xx,\yy) \in P^2$ such that $B(\xx,\yy) =
l$.
Note that $\sum_{l} w(l) = |P|^2-|\Pi_0|$, and since for any $\xx \in P$ there
exist at most $2q-1$ points $\yy \in \F_q^2$ such that $\|\xx-\yy\|=0$, we have
the bound $|\Pi_0|<2|P|q$. It can be assumed that $|P| \geq 4q$; this follows
from $|P|>q^{3/2}$ provided that $q \geq 16$, and for smaller values of $q$ the
theorem follows trivially by choosing suitably large constants hidden in the
$\gg$ notation. Since $|P| \geq 4q$, we have $$\sum_{l\in B(P)} w(l) =
|P|^2-|\Pi_0|>|P|^2/2.$$ Also, $\sum_{l} w(l)^2 = |Q|$.
By Cauchy-Schwarz, \[|P|^4 \ll \left ( \sum_{l \in B(P)} w(l) \right)^2 \leq
|B(P)| \sum w(l)^2  = |B(P)| |Q|.\] Hence, by Theorem \ref{energybound},
\[|B(P)| \gg \frac{|P|^4}{|P|^4/q^2 + q|P|^2}.\] Hence, if $|P| > q^{3/2}$, then
$|B(P)| \gg q^2$.
\end{proof}

Note that, as an alternative to Lemma \ref{th:distanceEnergyBound}, the bound
\begin{equation}
\sum_d|\Pi_d|^2 \ll |P|^{7/2}
\label{CSeasy}
\end{equation}
comes from a straightforward application of the Cauchy-Schwarz inequality. This is because 
\[\sum_d|\Pi_d|^2 \leq |P|^2\max_d|\Pi_d|,\]
and then the bound $\max_d|\Pi_d|=O(|P|^{3/2})$ can be obtained by constructing a set of circles of radius $d$ centred at points of $P$ and applying a Cauchy-Schwarz incidence bound to show that there are $O(|P|^{3/2})$ incidences between these circles and $P$.

If we plug this weaker bound into the proof of Lemma \ref{energybound2}, we obtain the bound 
\begin{equation}
|Q'| \ll \frac{|P|^{7/2}}{q}+q|P|^2.
\label{Qweak}
\end{equation}
Although this bound is not strong enough to prove Theorem  \ref{main1}, a careful look at the forthcoming analysis in section \ref{sec:pinnedDistances} will show that we can use \eqref{Qweak} instead of Lemma \ref{energybound2}. In particular, one can obtain the proof of Theorem \ref{main2} without needing to go through the extra work involved in proving Lemma \ref{th:distanceEnergyBound}.

\section{Application to pinned distances}\label{sec:pinnedDistances}

In this section, we use the bound on the bisector energy $Q$ to deduce an upper
bound on the number of isosceles triangles determined by a set of points in the
plane. A simple application of the Cauchy-Schwarz inequality translates this
into a lower bound on the number of pinned distances, proving Theorem
\ref{main2}.

One of the tools which will be needed is a weighted version of the Szemer\'{e}di-Trotter Theorem, which generalises ~\cite[Theorem~3]{Vinh} to the
case when the points and lines have multiplicity. 

\subsection*{A weighted version of the Szemer\'{e}di-Trotter Theorem}

Before stating and proving the incidence bound, let us first set up some
notation. Let $\mathcal{L}$ be a multiset of lines and let $\mathcal{P}$ be a
multiset of points in the plane $\mathbb{F}_q^2$. When considering the set of
lines in $\mathcal{L}$ or set of points in $\mathcal{P}$ without multiplicity,
we will refer to the set as $L$ or $P$, respectively. For a line $l\in{L}$, the
weight of $l$ is denoted $w(l)$, that is, $w(l)$ is the number of occurrences of
$l$ in the multiset $\mathcal{L}$.
Similarly, denote the weight of $p \in P$ by $w'(p)$.
Note that $$|\mathcal{L}|=\sum_{l\in{L}}w(l) \text{, and } |\mathcal{P}| =
\sum_{p\in{P}}w'(p) .$$

We define the number of incidences between $\mathcal{P}$ and $\mathcal{L}$ to be
$$I(\mathcal P,\mathcal{L})=\sum_{\xx\in{P}}\sum_{l\in{L}}w'(\xx)w(l)l(\xx).$$

\begin{lemma} \label{th:ST} Let $\mathcal{P}$ be a multiset of points in
$\FF_q^2$, and let $\mathcal{L}$ be a multiset of lines. Then $$I(\mathcal
P,\mathcal{L})\leq{\frac{|\mathcal
P||\mathcal{L}|}{q}+\left(\sum_{p\in{P}}w'(p)^2\right)^{1/2}\left(\sum_{l\in{L}}w(l)^2\right)^{1/2}q^{1/2}}.$$
\end{lemma}
\begin{proof}
The proof given here is identical to that in Vinh's article but with the $L^2$
expander mixing lemma instead of the traditional one. Each line in $\FF_q^2$ is
described by a point in the projective plane $\FF_q\PP^2$. Indeed any line $l$
is given by $l=\{(x,y)\in\FF_q^2:ax+by+c=0\}$ for some $(a,b,c)$
defined up to non-zero scalar multiples. We also have the usual embedding of $\FF_q^2$ into
the projective plane $(x,y)\mapsto [x:y:1]$. Consider the graph on
$q^2+q+1$ vertices given by points in $\FF_q\PP^2$, and which has as edges
\[E=\left\{\{[a:b:c],[x:y:z]\}:ax+by+cz=0\right\}.\] A straightforward
calculation shows that this graph is $(q+1)$-regular. After identifying
$l\in L$ and $p \in P$ with their corresponding points in $\FF_q\PP^2$, the number of weighted number of incidences is \[\sum_{l=[a:b:c]\in L}w(l)\sum_{\substack{p=[x:y:1]\in P\\ ax+by+c=0}}w'(p)=\langle w,Aw'\rangle\]
 Vinh showed that the
non-trivial eigenvalues of $A$ all have size at most $\sqrt q$ and thus the
lemma follows from Lemma \ref{th:expanderMixing}.
\end{proof}

\subsection*{Bounding the number of distinct isosceles triangles}The next task
is to use the weighted incidence bound to obtain an upper bound on the number of
isosceles triangles determined by $P$. The set of isoceles triangles determined
by $P$ is defined to be the set of ordered triples
$$\Delta(P):=\{(\xx,\yy,\zz)\in{P^3}:\|\xx-\zz\|=\|\yy-\zz\|, \|\xx-\yy\| \neq 0\}.$$
\begin{lemma} \label{triangles} For any set $P\subset{\FF_q^2}$,
$$|\Delta(P)|\ll{\frac{|P|^3}{q}+\frac{|P|^{5/2}}{q^{1/2}}+q|P|^{3/2}}.$$
\end{lemma}

\begin{proof} Define a multiset of lines $\mathcal{L}$ to be the set of
perpendicular bisectors determined by pairs of elements from $\xx,\yy \in P$
such that $\|\xx-\yy\| \neq 0$. The weight of a line $l\in{L}$ is the number of
pairs in $P\times{P}$ which determine $l$. That is,
$$w(l)=\{(\xx,\yy)\in{P\times{P}}:l=B(\xx,\yy)\}.$$

Now\footnote{In this particular incidence problem, we have a multiset of lines $\mathcal L$, but our set of points $P$ is not a multiset in the true sense. That is, all of the elements $p \in P$ have weight $w'(p)=1$.}, note that the number of weighted incidences $I(P,\mathcal{L})$ is precisely
the quantity $\Delta(P)$. Indeed, by the definition
of the perpendicular bisector $B(\xx,\yy)$, a point $\zz$ belongs to the line
$B(\xx,\yy)$ if and only if $(\xx,\yy,\zz)\in\Delta(P)$. Applying Lemma
\ref{th:ST} yields,
\begin{equation} \label{firstbound}
|\Delta(P)|\leq{\frac{|P||\mathcal{L}|}{q}+|P|^{1/2}\left(\sum_{l\in{L}}w^2(l)\right)^{1/2}q^{1/2}}.
\end{equation}

The quantity $|\mathcal{L}|$ is the total weight of the lines, which is the
number of pairs of elements of $P$ whose distance is non-zero. That is,
\begin{equation} \label{secondbound}
|\mathcal{L}|=|P|^2-|\Pi_0|<|P|^2.
\end{equation}

As in the proof of Theorem \ref{main1} the quantity $\sum_{l\in{L}}w^2(l)$ is
equal to $Q$, which was bounded in Theorem \ref{energybound}. Therefore, we have
\begin{equation} \label{thirdbound}
\sum_{l\in{L}}w^2(l)\ll{\frac{|P|^4}{q^2}+q|P|^2}.
\end{equation}
Combining \eqref{firstbound}, \eqref{secondbound} and \eqref{thirdbound}, it follows that
$$|\Delta(P)|\ll{\frac{|P|^3}{q}+\frac{|P|^{5/2}}{q^{1/2}}+q|P|^{3/2}},$$
as required.
\end{proof}
Note, in particular, that
\begin{equation}
|P| \geq q^{4/3} \Rightarrow |\Delta(P)|\ll \frac{|P|^3}{q}.
\label{specialcase}
\end{equation}

\begin{proof}[Proof of Theorem \ref{main2}]

Recall that Theorem \ref{main2} states that, if $|P| \geq q^{4/3}$ then there
exists a subset $P'$ such that $|P'|\gg |P|$ and, for all $\aa \in P'$,
$$|\{\|\aa-\bb\|:\bb\in P\}|\gg q.$$ Following a familiar argument from the
Euclidean pinned distances problem (see, for example \cite{PT}), it will be
shown that an upper bound on the number of isosceles triangles implies a lower
bound for the number of pinned distances.

For a point $\aa \in P$, construct a family of circles $\mathcal C_{\aa}$ which
consists of all circles centred at $\aa$ with non-zero radius which contain at
least one point from $P$. In particular, note that $|\mathcal
C_{\aa}|<|\{\|\aa-\bb\|:\bb \in P\}|$. Observe that
\begin{equation} \label{triangle}
|\Delta(P)|=\sum_{\aa \in P}\sum_{C\in \mathcal C_{\aa}}\left(|C\cap P|^2 - |\{(\bb,\cc) \in (C \cap P)^2 : \|\bb-\cc\|=0\}|\right).
\end{equation}

The next observation is that for a circle $C$ with non-zero radius centred at
$\aa$ and a fixed point $\bb \in C \cap P$, there is only one point which is a
distance zero from $\bb$ and lies on the circle $C$, and this point is $\bb$. Indeed, it can be verified directly that the only choice of $\cc$ satisfying the system of equations
\[ \|\cc-\bb\|=0, \,\,\,\,\, \|\cc-\aa\|=\|\bb-\aa\|\neq 0\]
is $\cc=\bb$.

Applying this information to \eqref{triangle} yields
\begin{align}
|\Delta(P)|&=\sum_{\aa \in P}\sum_{C\in \mathcal C_{\aa}}\left(|C\cap P|^2 - |C\cap P|\right)
\\&\geq \sum_{\aa \in P}\sum_{C\in \mathcal C_{\aa}}|C \cap P|^2-|P|^2. \label{eq2}
\end{align}
Note also that, since there are at most $2q-1$ points in $P$ which are a distance zero from a fixed point $\aa$,
\begin{align*}
\sum_{\aa \in P} \sum_{C \in \mathcal C_{\aa}}|C\cap P| & \geq \sum_{\aa \in P} (|P|-2q)
\\& \geq \sum_{\aa\in P}|P|/2
\\& \gg |P|^2
\end{align*}
In the second inequality it is assumed that $|P| \geq 4q$. This follows from the condition that $|P| \geq q^{4/3}$ provided that $q$ is sufficiently large, and for small $q$ the theorem is trivially true. Then, by Cauchy-Schwarz and \eqref{specialcase}
\begin{align*}
|P|^4&\ll \left(\sum_{\aa\in P}\sum_{C\in \mathcal C_{\aa}}|C \cap P| \right)^2
\\&\leq \left(\sum_{\aa}\sum_{C\in{\mathcal C_{\aa}}}|C \cap P|^2\right) \left(\sum_{\aa \in P}|\{\|\bb-\aa\|:\bb\in P\}|\right)
\\&\ll  \left(|\Delta(P)|+|P|^2\right) \left(\sum_{\aa \in P}|\{\|\bb-\aa\|:\bb\in P\}|\right)
\\&\ll  \left(\frac{|P|^3}{q}\right) \left(\sum_{\aa \in P}|\{\|\bb-\aa\|:\bb\in P\}|\right).
\end{align*}
Therefore,
\begin{equation}
\sum_{\aa \in P}|\{\|\bb-\aa\|:\bb\in P\}| \geq 2c|P|q.
\label{case11}
\end{equation}
for some constant $c>0$.

Now, define $P':=\{\aa \in P:|\{\|\bb-\aa\|:\bb\in P\}| \geq cq\}$. It remains
to show that $|P'| \gg |P|$. This follows from \eqref{case11}, since
\begin{align*}
2c|P|q &\leq \sum_{\aa \in P}|\{\|\bb-\aa\|:\bb\in P\}| 
\\&= \sum_{\aa \in P'}|\{\|\bb-\aa\|:\bb\in P\}| +  \sum_{\aa \in P \setminus P'}|\{\|\bb-\aa\|:\bb\in P\}|
\\&\leq \sum_{\aa \in P'}|\{\|\bb-\aa\|:\bb\in P\}| +  c|P|q,
\end{align*}
which implies that
$$c|P|q \leq  \sum_{\aa \in P'}|\{\|\bb-\aa\|:\bb\in P\}| \leq |P'|q,$$
and thus $|P'|\gg |P|$.
\end{proof}

\section*{Acknowledgements} Brandon Hanson was supported by NSERC of Canada. Ben
Lund was supported by NSF grant CCF-1350572.
Oliver Roche-Newton was supported by the Austrian Science Fund (FWF): Project
F5511-N26, which is part of the Special Research Program ``Quasi-Monte Carlo
Methods: Theory and Applications".
Part of this research was undertaken when the authors were visiting the
Institute for Pure and Applied Mathematics, UCLA, which is funded by the NSF. We
are grateful to Swastik Kopparty, Doowon Koh, Tom Robbins, Adam Sheffer and
Frank de Zeeuw for several helpful conversations related to the content of this
paper. Finally, we are grateful to an anonymous referee for several comments which have helped to improve the exposition of the paper.

\appendix
\section{Facts from finite plane geometry}\label{sec:GeometryAppendix}

All of the results in this appendix are quite elementary and rely only on linear
algebra over finite fields. However, in the interest of self-containment we
have included them.

A rigid motion of $\FF_q^2$ is an affine map $\cA(\xx)=A\xx+\bb$ such that
\[\|\xx-\yy\|=\|\cA(\xx)-\cA(\yy)\|.\] Thus rigid motions map a circle of a
given radius to another circle of the same radius. We will be interested in
rotations, reflections and translations, which were defined in section
\ref{sec:finitePlaneGeometry}. 

A rotation is said to be trivial if its corresponding rotation matrix is the
identity, whilst a translation by $\uu$ is said to be trivial if $\uu=0$. Note that there are no trivial
reflections. Also note that the product of two rotation matrices or two
reflection matrices is a rotation matrix. An obvious first remark is that
rotation and reflection matrices are unitary so that rotations, reflections and
translations are in fact rigid.

The rigid motions we are working with are essentially described by their fixed
points:

\begin{applemma}
Any non-trivial translation has no fixed points. Any non-trivial rotation has a
unique fixed point. Any reflection has a unique fixed affine line.
\end{applemma}
\begin{proof}
That a non-trivial translation fixes no points is clear. 

Suppose we have a rotation $\cR_{R,\uu}$ with $R\neq I$. Then $\uu$ is clearly
fixed. If $\vv$ was also fixed we would have $(R-I)\vv=(R-I)\uu$ (where $I$ is
the identity). However, $\det(R-I)$ is non-zero\footnote{Note that we use
the assumption that the characteristic of $\F_q$ is not equal to 2 in this
calculation.} since $R\neq I$. Thus $R-I$ is invertible and $\uu=\vv$.

A reflection matrix $S\neq I$ has eigenvalues $\pm 1$. If $\cS_{S,\uu}$ fixes
$\vv$ then $\uu-\vv$ lies in the eigenspace of $1$, which is a line $l$. Hence, $\vv\in l+\uu$.
\end{proof}

Two rotations $\cR_{R_1,\uu_1}$ and $\cR_{R_2,\uu_2}$ are called
\textit{complimentary} if $R_1^{-1}=R_2$. Two reflections $\cS_{S_1,\uu_1}$ and
$\cS_{S_2,\uu_2}$ are called \textit{parallel} if $S_1=S_2$. It is a
straightforward computation that the fixed lines of two parallel reflections are parallel.

\begin{applemma}\label{th:composition} The composition of any two
non-complimentary rotations is a rotation, while the composition of two
complimentary rotations is a translation. The composition of any two
non-parallel reflections is a rotation while the composition of any two parallel
reflections is a translation. The composition of a non-trivial rotation and a
translation is a rotation.
\end{applemma}
\begin{proof}
Suppose we have rotations by $R_1$ and $R_2$ about $\uu_1$ and $\uu_2$ respectively. The composition 
is the map \[\vv\mapsto R_2R_1\vv-R_2R_1\uu_1+R_2\uu_1-R_2\uu_2+\uu_2.\] This is a rotation provided 
there is a $\uu$ such that \[(R_2R_1-I)\uu=R_2R_1\uu_1-R_2\uu_1+R_2\uu_2-\uu_2\] which exists provided 
$R_2R_1$ is not the identity. If it is the identity then the
rotations are complimentary and the composition is a translation. 
The same proof works when $R_1$ and $R_2$ are replaced by reflection matrices as
the product of two reflection matrices is a rotation matrix.

If we translate by $\uu_1$ and then rotate by $R$ about $\uu_2$ then the
composition is  \[\vv\mapsto R(\vv+\uu_1)-R\uu_2+\uu_2.\] To show this is a
rotation it suffices to find $\uu$ such that \[(R-I)\uu=R\uu_2-R\uu_1-\uu_2\]
which can be done since $R-I$ is invertible.
\end{proof}

Our primary reason for being interested in rigid motions is the relationship
between reflections and their fixed lines. Observe that when $\uu'$ lies on the
fixed line of a reflection $\cS_{S,\uu}$ then $\cS_{S,\uu}=\cS_{S,\uu'}$. 

\begin{applemma}\label{th:LinesAndReflections}
Suppose $\xx\in \FF_q^2$ and $\cS$ is a reflection which does not fix $\xx$. Then the fixed line of
$\cS$ is $B(\xx,\cS(\xx))$. Moreover, a line $l$ is the fixed line of a
unique reflection if and only if it is non-isotropic. If $\yy\in \FF_q^2$ is any
point such that $\|\xx-\yy\|\neq 0$, then the line $B(\xx,\yy)$ is non-isotropic
and there is a unique reflection $\cS$ such that $\cS(\xx)=\yy$ which fixes it. 
\end{applemma}
\begin{proof}
Observe that if $\uu$ is fixed by $\cS$ then
\[\|\xx-\uu\|=\|\cS(\xx)-\cS(\uu)\|=\|\cS(\xx)-\uu\|\] so that $\uu$ lies on
$B(\xx,\cS(\xx))$. But the fixed points of $\cS$ form a line, so it must
coincide with $B(\xx,\cS(\xx))$. 

Let $\uu_1$ and $\uu_2$ be any distinct points on the line $l$, which is assumed to
be non-isotropic. Set $\dd=(d_1,d_2)=\uu_1-\uu_2$ so that $\|\dd\|\neq 0$. The
reflection $\cS$ by
\[\frac{1}{d_1^2+d_2^2}\mat{d_1^2-d_2^2}{2d_1d_2}{2d_1d_2}{d_2^2-d_1^2}\] about
$\uu_1$ fixes $l$. This reflection is in fact unique. If $\cS'$ were another
reflection fixing $l$ then their composition would be either a rotation or a
translation fixing a line. It follows that $\cS'=\cS$.

Finally, suppose $\yy\in\FF_q^2$ is distinct from $\xx$. If $B(\xx,\yy)$ is
isotropic then for distinct points $\uu_1,\uu_2\in B(\xx,\yy)$ we have
$\dd=\uu_1-\uu_2=t\cdot (1,\pm i)$. Since $\xx-\yy$ is orthogonal to $\dd$ we
must have that $t(\xx\pm i\yy)=0$ and $\|\xx-\yy\|=0$.
\end{proof}

\begin{proof}[Proof of Lemma \ref{th:distanceDecomposition}]
Since $B$ is non-isotropic, there is a unique reflection $\cS$ which fixes it.
Then $\zz=\cS(\xx)$ and $\ww=\cS(\yy)$ and the result follows by rigidity.
\end{proof}

We have already mentioned that rigid motions send circles to circles. In fact
given two points on a circle, we now discuss when a rigid motion sends one to
the other.

\begin{applemma}
Let $\xx,\yy\in C_r(\uu)$ for elements $\xx,\yy,\uu\in\FF_q^2$ and $r\neq 0$.
There is a unique rotation $\cR$ fixing $\uu$ and sending $\xx$ to
$\yy$.
\end{applemma}
\begin{proof}
After applying a translation if necessary we can assume $\uu=0$. Then we have
points $\xx$ and $\yy$ with $\|\xx\|=\|\yy\|$. We need a rotation matrix $R$
such that $R\xx=\yy$. That is, we are to solve
\[\mat{a}{-b}{b}{a}\vec{x_1}{x_2}=\vec{y_1}{y_2}\] where $\xx=(x_1,x_2)$ and $\yy=(y_1,y_2)$.
This is the same as solving
\[\mat{x_1}{-x_2}{x_2}{x_1}\vec{a}{b}=\vec{y_1}{y_2}.\] Since $x_1^2+x_2^2\neq0$ this 
linear equation has a unique solution. We need to check that this unique solution satisfies $a^2+b^2=1$. 
Indeed, we can write
\begin{equation}
\vec{a}{b}=\|\xx\|^{-1}\mat{x_1}{x_2}{-x_2}{x_1}\vec{y_1}{y_2}=\|\xx\|^{-1}\vec{x_1y_1+x_2y_2}{x_1y_2-x_2y_1}
\label{reflectioncalculation}
\end{equation}
It follows that
\[\|(a,b)\|=\|\xx\|^{-2}\|(x_1y_1+x_2y_2,x_1y_2-x_2y_1)\|=\|\xx\|^{-2}\|\xx\|\|\yy\|=1.\]
Suppose there were another rotation matrix $R'$ with the same property. Then
$R'R^{-1}$ would be a rotation matrix fixing two points, $0$ and $\yy$, and so
must be the identity. 
\end{proof}

\begin{applemma} \label{th:QuadruplesAndRotations}
Suppose $\xx,\yy,\zz,\ww\in\FF_q^2$ such that $(\xx,\yy) \neq (\zz,\ww)$ and \[\|\xx-\yy\|=\|\zz-\ww\|\neq 0.\] If
$\xx-\yy\neq \zz-\ww$, then there is a unique rotation $\cR$ with
$\cR(\xx)=\zz$ and $\cR(\yy)=\ww$, and there are no translations with this
property. If $\xx-\yy=\zz-\ww$ then there is a unique translation $\cT$ with
$\cT(\xx)=\zz$ and $\cT(\yy)=\ww$, and there are no rotations with this property.
\end{applemma}
\begin{proof}
Let $\cT$ be translation by $\zz-\xx$. Then $\cT(\xx)=\zz$. If
$\xx-\yy=\zz-\ww$ then $\cT(\yy)=\ww$ and $\cT$ is the desired translation and
it is plainly unique. Moreover, if $\cR$ is a rotation with $\cR(\xx)=\zz$ and
$\cR(\yy)=\ww$ then by Lemma \ref{th:composition}, $\cR^{-1}\circ \cT$ is a
non-trivial rotation fixing both $\xx$ and $\yy$ which is impossible.

Otherwise $\xx-\yy\neq\zz-\ww$ and
\[\|\zz-\ww\|=\|\xx-\yy\|=\|\cT(\xx)-\cT(\yy)\|=\|\zz-\cT(\yy)\|.\] Thus
$\cT(\yy)$ and $\ww$ lie on a common circle centered at $\zz$ and there is a
non-trivial rotation $\cR$ about $\zz$ with $\cR(\cT(\yy))=\ww$. Then by Lemma
\ref{th:composition}, $\cR'=\cR\circ \cT$ is the desired rotation. As for
uniqueness, if we had another non-trivial rotation $\cR''$ then the non-trivial
rotation $\cR'^{-1}\circ\cR''$ would fix both $\xx$ and $\yy$ which is
impossible. Similarly, if $\cT$ is a translation with $\cT(\xx)=\zz$ and
$\cT(\yy)=\ww$ then $\cR'^{-1}\circ \cT$ is a non-trivial rotation fixing both
$\xx$ and $\yy$ which is impossible.
\end{proof}

We are now in a position to prove the necessary lemmas from section
\ref{sec:finitePlaneGeometry}.

\begin{proof}[Proof of Lemma \ref{th:reflections}]
After a translation we can assume $\uu=0$. Then any point $\vv\in \FF_q^2$ lies
on a line passing through $0$. Each of these lines contains exactly
$q-1$ non-zero points and hence there are $\frac{q^2-1}{q-1}=q+1$ lines passing
through $0$. 

If $q=1\mod 4$ then the lines through $0$ which are isotropic are spanned by
$(1,i)$ or $(1,-i)$ and the rest are non-isotropic. The non-isotropic lines are
fixed by a unique reflection about $0$. Since rotations and reflections are in bijection, there are $q-1$ of each about $0$.

When $q=3\mod 4$ there are no elements of zero norm and so no isotropic lines.
The rest of the analysis is as in the $q=1\mod 4$ case
\end{proof}

\begin{proof}[Proof of Lemma \ref{th:circles}]
Again we may translate and assume $\uu=0$. Any point $\vv \in C_0(0)$ satisfies
$\|\vv\|=0$. When $q = 1\mod 4$ all such $\vv$ are of the form $(v,\pm iv)$ so
that there are $2q-1$ or them. When $q=3\mod 4$ only $\vv=0$ is possible. Thus we are left with
$r\neq 0$. In this case fix any $\vv \in C_r(0)$. For any other $\ww\in C_r(0)$
there is a unique rotation about $0$ taking $\vv$ to $\ww$. When $q=1\mod 4$
there are $q-1$ such rotations and when $q=3\mod 4$ there are $q+1$ such
rotations. The second claim of the theorem is immediate.
\end{proof}

\begin{applemma}Any non-trivial rotation can be decomposed into a pair of
reflections:
\begin{enumerate}
  \item when $q=1\mod 4$ there are $q-1$ decompositions;
  \item when $q=3\mod 4$ there are $q+1$ decompositions.
\end{enumerate}
Any translation by $\dd$ with $\|\dd\|\neq 0$ can be decomposed into
a pair of reflections in $q$ ways.
A non-trivial translation by an isotropic vector cannot be decomposed into a pair of reflections.
\end{applemma}
\begin{proof}
Suppose we have a rotation $\cR=\cR_{R,\uu}$.
For any reflection matrix $S$, $RS^{-1}$ is also a reflection matrix. Then
$\cR=\cS_{RS^{-1},\uu}\circ\cS_{S,\uu}$, and since $S$ could be any of the
reflection matrices, there are at least $q-1$ such decompositions when $q=1\mod 4$ and $q+1$ when $q=3\mod 4$. We now prove that this accounts for all such decompositions. If
$\cR=\cS_1\circ \cS_2$ then $\cS_1$ and $\cS_2$ are non-parallel for otherwise their composition would be a
translation. The reflections $\cS_1$ and $\cS_2$ have fixed lines $l_1$ and $l_2$ which are
non-parallel and so intersect at a point $\vv$. This point is fixed by $\cR$
and is uniquely so since $\cR$ is non-trivial, that is $\vv=\uu$. The reflection
matrices $S_1$ of $\cS_1$ and $S_2$ of $\cS_2$ then have to satisfy
$S_2=RS_1^{-1}$ as required.

Now suppose we have a translation by a non-isotropic element $\dd$. Let $S$ be
the reflection matrix such that
$S\dd=-\dd$. Then if $t\in \FF_q$, the composition of \[\cS_1(\vv)=
S\vv-S(-t\dd)+(-t\dd)\] and
\[\cS_2(\vv)=S\vv-S\left(\left(\frac{1}{2}-t\right)\dd\right)+\left(\frac{1}{2}-t\right)\dd\]
is translation by $\dd$. This gives us $q$ distinct decompositions; indeed the
fixed line of the reflection by $\cS_1$ is orthogonal to $\dd$ and passes
through $-t\dd$ and is therefore distinct for distinct values of $t$.
Now, suppose $\cS_1$ and $\cS_2$ are two reflections which compose to
translation by $\dd$. Then $\cS_1$ and $\cS_2$ are parallel with common
reflection matrix $S$ about points $\uu_1$ and $\uu_2$ respectively. Since
\[\vv+\dd=\cS_2\circ\cS_1(\vv)=S(S\vv-S\uu_1+\uu_1)-S\uu_2+\uu_2=\vv+S(\uu_1-\uu_2)-(\uu_1-\uu_2)\]
we see that $S\dd=-\dd$ and so the decomposition is of the form we described.

Now, we observe that all pairs of reflections are now accounted for, and hence
there is no way to decompose translation by a non-zero isotropic vector into a
pair of reflections.
From Lemma \ref{th:reflections}, it is clear that (in the $q = 1 \mod 4$ case)
there are in total $(q-1)q$ reflections, $(q-2)q^2$ non-identity rotations, and
$(q-1)^2$ translations by a non-zero distance.
Hence, there are $(q-1)^2q^2$ pairs of reflections, of which $(q-2)q^2(q-1)$ are
non-identity rotations, $(q-1)^2 q$ are translations by a non-zero distance, and
$(q-1)q$ are the identity.
We have that $(q-1)^2q^2 = (q-2)q^2(q-1) + (q-1)^2q + (q-1)q$.
\end{proof}

\begin{proof}[Proof of Lemma \ref{th:rotationsFromReflections2}]
Suppose first that $\xx-\yy=\zz-\ww$. Then $\xx\neq \zz$ for otherwise
$\yy=\ww$. Thus translation by $\dd=\zz-\xx$ is the unique translation $\xx$ to
$\zz$ and $\yy$ to $\ww$. This translation can be decomposed in $q$ ways if $\dd$ is not isotropic, and $0$ ways if $\dd$ is isotropic. There
are no other pairs of reflections with the desired property.  Indeed, by Lemma
\ref{th:composition} and Lemma \ref{th:QuadruplesAndRotations}, if two
reflections have the desired property, then the composition of these two 
reflections must be a translation, and so it must be the unique translation 
taking $\xx$ to $\zz$ and $\yy$ to $\ww$.

If $\xx-\yy\neq\zz-\ww$ then there is a unique rotation taking $\xx$ to $\zz$
and $\yy$ to $\ww$. This can be decomposed in $q-1$ ways when $q=1\mod 4$ and
in $q+1$ ways when $q=3\mod 4$.
Similarly to the above, by Lemma \ref{th:composition} and Lemma
\ref{th:QuadruplesAndRotations}, there are no other pairs of reflections with
the desired property.
\end{proof}

\section{Proof of the $L^2$ Expander Mixing Lemma} \label{app2}

Here we give a proof of our weighted Expander Mixing Lemma, which is essentially
the same as the standard proof.

\begin{proof}[Proof of Lemma \ref{th:expanderMixing}]
Write \[f=\sum_{e}\langle f,e\rangle e\] and \[Ag=\sum_{e}\langle Ag,e\rangle
e=\sum_{e}\lambda_e\langle g,e\rangle
e\] where the summation is over the eigenfunctions $e$ of $A$ with eigenvalues
$\lambda_e$.
Via Lemma \ref{Fourier} we have \[\langle f,Ag\rangle =\sum_e
\lambda_e\langle f,e\rangle\overline{\langle
g,e\rangle}=\delta n\EE(f)\EE(g)+\sum_{e\neq e_1} \lambda_e\langle
f,e\rangle\overline{\langle g,e\rangle}.\] Here we have extracted the
contribution from the constant function $e_1(v)=1/\sqrt n$. Since $|\lambda_e|\leq \lambda$
for each $e\neq e_1$, after an application of Cauchy-Schwarz and Lemma \ref{Fourier} we see 
\begin{align*}
\left|\langle f,Ag\rangle-\delta n\EE(f)\EE(g)\right|&\leq \lambda\sum_{e\neq
e_1}|\langle f,e\rangle||\langle g,e\rangle|\\
&\leq \lambda\left(\sum_e |\langle
f,e\rangle|^2\right)^\frac{1}{2}\left(\sum_e |\langle
f,e\rangle|^2\right)^\frac{1}{2}\\
&=\lambda\|f\|\|g\|.
\end{align*} The second claim follows by taking $f$ and $g$ to be the
characteristic functions of $S$ and $T$ respectively.
\end{proof}

\end{document}